\documentclass[11pt]{article}

\usepackage{geometry}                
\usepackage{graphicx}
\usepackage{xcolor}
\usepackage{amssymb,amsmath,amsthm}
\usepackage{amsmath,amssymb,amsthm,amsfonts,amstext,amsbsy,amscd}
\usepackage[utf8]{inputenc}
\usepackage{bbm,dsfont}
\usepackage{color}
\usepackage{comment}
\numberwithin{equation}{section}
\theoremstyle{plain}
\newtheorem{prop}{Proposition}

\newtheorem{theo}{Theorem}
\newtheorem{lem}{Lemma}

\newtheorem{ex}{Example}
\numberwithin{equation}{section}
\numberwithin{lem}{section}
\numberwithin{prop}{section}
\numberwithin{theo}{section}
\numberwithin{cor}{section}
\usepackage{epstopdf}
\usepackage{authblk}

\title{Nonparametric local polynomial regression for functional covariates}

\author{Moritz Jirak \\ \vspace*{-.3cm}{\footnotesize  Institut für Statistik und Operations Research, Universit\"at Wien,  Oskar-Morgenstern-Platz 1, 1090 Wien, Austria \\ Email: moritz.jirak@univie.ac.at}  \and Alois Kneip\\ \vspace*{-.3cm}{\footnotesize Fachbereich Wirtschaftswissenschaften, Universit\"at Bonn, Adenauerallee 24 -- 42, 53113 Bonn, Germany  \\ Email: akneip@uni-bonn.de} \and Alexander Meister\\ \vspace*{-.3cm}{\footnotesize Institut f\"ur Mathematik, Universit\"at Rostock, Ulmenstra{\ss}e 69, 18057 Rostock, Germany  \\ Email: alexander.meister@uni-rostock.de} \and Mario Pahl\\ \vspace*{-.3cm}{\footnotesize Institut f\"ur Mathematik, Universit\"at Rostock, Ulmenstra{\ss}e 69, 18057 Rostock, Germany  \\ Email: mario.pahl@uni-rostock.de}}

\date{}                                           

\begin{document}

\maketitle

\begin{abstract}
We consider nonparametric regression with functional covariates, that is, they are elements of an infinite-dimensional Hilbert space. A locally polynomial estimator is constructed, where an orthonormal basis and various tuning parameters remain to be selected. We provide a general asymptotic upper bound on the estimation error and show that this procedure achieves polynomial convergence rates under appropriate tuning and supersmoothness of the regression function. Such polynomial convergence rates have usually been considered to be non-attainable in nonparametric functional regression without any additional strong structural constraints such as linearity of the regression function.
\end{abstract}

\vspace{1cm} 


\section{Introduction}

We consider the nonparametric regression model where one observes the i.i.d.~data $(X_j,Y_j)$, $j=1,\ldots,n$, given as
\begin{equation} \label{eq:model} Y_j \, = \, g(X_j) + \varepsilon_j\,. \end{equation}
The covariates $X_j$ are functional random variables and take on their values in the Hilbert space $L_2([0,1])$, which is endowed with its Borel $\sigma$-algebra. The regression errors $\varepsilon_j$ are conditionally centered given $X_j$ and have a finite conditional variance given $X_j$, which is bounded from above by some deterministic constant $\sigma^2>0$. The goal is to estimate the smooth regression function $g$, which maps from $L_2([0,1])$ to $\mathbb{R}$. 

A famous simplification of this model is functional linear regression where $g$ is imposed to be linear, i.e.~$g(x) = \langle \theta , x\rangle$ with $L_2([0,1])$-inner product $\langle\cdot,\cdot\rangle$ and  target function $\theta \in L_2([0,1])$. Then the statistical problem reduces to the estimation of $\theta$. The model has been widely studied, e.g.~for nonparametric estimation of $\theta$ with optimal convergence rates, we refer to \cite{HH07}, to \cite{CJ12} for adaptive estimation, to \cite{CH06} for optimal prediction, to \cite{M11} for asymptotic equivalence to a white noise inverse problem and to \cite{CMS07} for central limit theorems. Despite the linear structure, polynomial rates slower than the standard parametric rate occur since the covariates take their values in an infinite-dimensional function space.  

For approaches beyond linearity, we mention quadratic functional regression, see \cite{YM10} and fully nonparametric functional regression models (see e.g.~the book of \cite{FV06}, the review paper of \cite{LV18} and the references therein). While convergence rates are usually not provided explicitly, polynomial rates have not been attained in the literature so far. The most commonly used estimation method is the functional data version of the kernel regression estimator (Nadaraya-Watson estimator). Besides, there are extensions from this locally constant estimator to locally linear procedures, see e.g.~\cite{BG09}, \cite{BFV10} and \cite{BEM11}.

In the current work, we introduce a locally polynomial estimator for an arbitrary degree. It generalizes the standard local polynomial regression estimator (see e.g.~\cite{FG96}) to treat functional covariates. We mention approaches to nonparametric density estimation for functional data, i.e.~the wavelet-based method of \cite{CKM13} and orthogonal series and kernel estimators by \cite{DN04a} and \cite{DN04b} for specific diffusion processes. While upper bounds on the convergence rates are provided under quite abstract conditions in these papers, \cite{DM21} study artificially contaminated functional data with a Wiener density and attain polynomial convergence rates by series estimators. Under finite smoothness levels of the target function in functional data analysis, only logarithmic or at least sub-polynomial convergence rates are achieved (for lower bounds see \cite{M12},\cite{CR14} and \cite{M16}). However, the infinite dimensional Gaussian mixtures considered in \cite{DM21} are supersmooth in a sense that the functional densities are infinitely-fold differentiable. In the current paper we show that, for functional regression functions under more general supersmoothness constraints (apart from Gaussian mixtures), polynomial convergence rates can be attained by our locally polynomial estimator under specific selection of an orthonormal basis and two tuning parameters denoted by $J$ and $K$. 

In Section \ref{Meth}, we introduce our estimation procedure, in Section \ref{GAUB} we deduce a general asymptotic upper bound on the pointwise estimation error of our method in Theorem \ref{T:1}. This results is used in Section \ref{PCR} in order to derive polynomial convergence rates. The proofs are deferred to Section \ref{Proofs}.

\section{Methodology} \label{Meth}

Let us assume that $g$ is $K$-fold Fr{\'e}chet differentiable in some neighbourhood ${\cal N} := {\cal N}(x,\delta)$ of $x \in L_2([0,1])$ with radius $\delta>0$ where $x$ is the site at which $g$ should be estimated. The $K$th order Fr{\'e}chet derivative is supposed to be continuous on ${\cal N}$. We define
$$ \tilde{g}_j(\lambda) \, := \, g\big(\lambda X_j + (1-\lambda) x \big)\,. $$
Under the stipulation that $X_j \in {\cal N}$, we apply Taylor approximation to $\tilde{g}$ around $\lambda=0$, which yields that
\begin{align} \nonumber 
 g(X_j) & \, = \, \tilde{g}_j(1) = \sum_{k=0}^{K-1} \frac1{k!} \tilde{g}_j^{(k)}(0) \, + \, {\cal R}_S^{(K)}(x,X_j) \\ \label{eq:dimapp} & \, = \, \sum_{k=0}^{K-1} \frac1{k!} g^{(k)}(x;X_j-x,\ldots,X_j-x) \, + \, {\cal R}_S^{(K)}(x,X_j)\,, \end{align}
where $g^{(k)}(x;\cdots)$ denotes the $k$th order Fr{\'e}chet derivative of $g$ at $x$, which is viewed as a symmetric multilinear map of degree $k$; and the remainder term satisfies
\begin{equation} \label{eq:remain} \big|{\cal R}_S^{(K)}(x,X_j)\big| \, \leq \, \frac1{K!} \cdot \sup_{y \in {\cal N}} \big\|g^{(K)}(y;\cdot)\big\| \cdot \delta^K\,, \end{equation}
with
\begin{equation} \label{eq:gnorm} \big\|g^{(K)}(y;\cdot)\big\| \, := \, \sup \big\{ \big|g^{(K)}(y;u_1,\ldots,u_K)\big| \, : \, \|u_1\| = \cdots = \|u_K\| = 1\big\}\,. \end{equation}
Let $\{\varphi_j\}_{j\in \mathbb{N}}$ be some orthonormal basis of $L_2([0,1])$, which is considered as deterministic at this stage. We replace $X_j-x$ by its finite-dimensional proxy $\sum_{l=1}^J \langle X_j-x,\varphi_l\rangle \varphi_l$, for some integer $J>0$, where $\langle\cdot,\cdot\rangle$ stands for the $L_2([0,1])$-inner product, so that
\begin{align}  \label{eq:app}
 g(X_j) & =  {\cal P}_{J,K}\big(\langle X_j-x,\varphi_1\rangle , \ldots,  \langle X_j-x,\varphi_J\rangle \big)   +  {\cal R}_D^{(K,J)}(x,X_j) +  {\cal R}_S^{(K)}(x,X_j)\,,
 \end{align}
where 
\begin{align} \nonumber {\cal P}_{J,K} \big(\langle X_j-x,\varphi_1\rangle , \ldots,   \langle X_j-x,&\varphi_J\rangle \big) \\ \nonumber & =  \sum_{k=0}^{K-1} \frac1{k!} \sum_{l_1=1}^{J} \cdots \sum_{l_k=1}^{J} g^{(k)}(x;\varphi_{l_1},\ldots,\varphi_{l_k}) \cdot \prod_{q=1}^k \langle X_j-x,\varphi_{l_q}\rangle \label{eq:tayapp}\\
&  =  \sum_{k\in {\cal K}} \xi_{j;{\bf k}} \cdot G_{\bf k}(x)\,,
\end{align}
with the notation
\begin{align} \nonumber
{\cal K} & \, := \, \big\{{\bf k}=(k_1,\ldots,k_J) \in \mathbb{N}_0^J \, : \, k_1+\cdots+k_J \leq K-1\big\}\,, \\ \nonumber
\xi_{j;{\bf k}} & \, := \, \prod_{l=1}^J \langle X_j-x,\varphi_l\rangle^{k_l}\,, \\ \label{eq:not}
G_{\bf k}(x) & \, := \, g^{(|{\bf k}|)}(x;\underbrace{\varphi_1,\ldots,\varphi_1}_{k_1-\mbox{fold}},\ldots,\underbrace{\varphi_J,\ldots,\varphi_J}_{k_J-\mbox{fold}}) \, \big/ \, \prod_{j=1}^J k_j!\,,
\end{align}
where $|{\bf k}| = \sum_{j=1}^J k_j$. Note that ${\cal P}_{J,K}$ forms a $J$-variate polynomial of degree $\leq K-1$ and that ${\cal R}_D^{(K,J)}(x,X_j)$ represents the error coming from the finite-dimensional approximation. It can be bounded from above as follows
\begin{lem} \label{L:bounddim}
On the event $X_j\in {\cal N}$ it holds that 
$$ \big|{\cal R}_D^{(K,J)}(x,X_j)\big| \, \leq \, \sum_{k=1}^{K-1} \frac{\delta^{k-1}}{(k-1)!} \cdot \big\|g^{(k)}(x;\cdot)\big\| \cdot \Big\{\sum_{l>J} \langle X_j-x,\varphi_l\rangle^2\Big\}^{1/2}\,. $$ 
\end{lem}
As the estimation procedure we suggest to consider that polynomial $\hat{{\cal P}}_{J,K}$ which minimizes the contrast functional
$$ \tilde{{\cal P}}_{J,K} \, \mapsto \, \sum_{j=1}^n \big|Y_j - \tilde{{\cal P}}_{J,K}\big(\langle X_j-x,\varphi_1\rangle , \ldots,  \langle X_j-x,\varphi_J\rangle \big)\big|^2 \cdot 1_{{\cal N}}(X_j)\,, $$
among all $J$-variate polynomials $\tilde{{\cal P}}_{J,K}$ of the degree $\leq K-1$. Following the standard strategy in local polynomial regression, the constant coefficient of $\hat{{\cal P}}_{J,K}$ can be employed to estimate $g(x)$.
The scheme is formalized as follows: Write $\|z\|_{\cal N}^2 := \sum_{j=1}^n z_j^2 \cdot 1_{{\cal N}}(X_j)$, $Y := (Y_1,\ldots,Y_n)^\dagger$ and
$$ \tilde{{\cal P}}_{J,K}\big(\langle X_j-x,\varphi_1\rangle , \ldots,  \langle X_j-x,\varphi_J\rangle\big)  = \sum_{{\bf k} \in {\cal K}}\, \alpha_{{\bf k}} \cdot \xi_{j;{\bf k}}\,. $$
Then the goal is to select $\alpha = (\alpha_{{\bf k}})_{{\bf k}}$ such that $\|Y - \xi \alpha\|_{{\cal N}}^2$ with $\xi:=\{\xi_{j,{\bf k}}\}_{j=1,\ldots,n;{\bf k}\in {\cal K}}$ is minimized. By usual least-square arguments, we deduce that $\hat{\alpha}$ is a minimizing vector if it satisfies the following system of linear equations
$$ {\cal M} \hat{\alpha} \, = \, {\cal Y}\,, $$
where 
\begin{align*} 
{\cal Y} & \, = \, \Big(\sum_{j=1}^n 1_{\cal N}(X_j) \xi_{j;{\bf k}}\cdot Y_j\Big)_{{\bf k}\in {\cal K}}\,, \\ 
{\cal M} & \, = \, \Big(\sum_{j=1}^n 1_{\cal N}(X_j)\cdot \xi_{j;{\bf k}}\cdot \xi_{j;{\bf k}'}\Big)_{{\bf k},{\bf k}'\in {\cal K}}\,. \end{align*}
This motivates the definition of the functional local polynomial estimator
\begin{equation} \label{eq:hatg}
\hat{g}(x) \, = \, {\bf e}_0^\dagger ({\cal M} + {\cal S})^{-1} {\cal Y}\,,
\end{equation}
for $j$th unit vector ${\bf e}_j$ and the diagonal ${\cal K}\times {\cal K}$-matrix ${\cal S}$ whose $({\bf k},{\bf k})$th entry equals $1/{|{\bf k}| \choose k_1,\ldots,k_J}$. This Tikhonov regularization guarantees invertibility of ${\cal M}+{\cal S}$ as ${\cal M}$ is symmetric and positive semi-definite. 

\section{General Asymptotic Upper Bound} \label{GAUB}

The error of the estimator (\ref{eq:hatg}) is decomposed as follows
\begin{equation} \label{eq:errt}
\hat{g}(x) - g(x) \, = \, {\cal B}_1 \, + \, {\cal B}_2 \, + \, {\cal B}_3 \, + \, {\cal V}\,,
\end{equation}
where 
\begin{align*}
{\cal B}_1 & \, := \, - {\bf e}_0^\dagger ({\cal M}+{\cal S})^{-1}\, {\cal S}\, {\bf G}\,, \\ 
{\cal B}_2 & \, := \, {\bf e}_0^\dagger ({\cal M}+{\cal S})^{-1}\, \sum_{j=1}^n 1_{\cal N}(X_j) \cdot \xi_{j;\bullet} \cdot {\cal R}_D^{(K,J)}(x,X_j)\,, \\
{\cal B}_3 & \, := \,  {\bf e}_0^\dagger ({\cal M}+{\cal S})^{-1}\, \sum_{j=1}^n 1_{\cal N}(X_j) \cdot \xi_{j;\bullet} \cdot {\cal R}_S^{(K)}(x,X_j)\,, \\
{\cal V} & \, := \,  {\bf e}_0^\dagger ({\cal M}+{\cal S})^{-1}\, \sum_{j=1}^n 1_{\cal N}(X_j) \cdot \xi_{j;\bullet} \cdot \varepsilon_j\,,
\end{align*}
since $G_0 = g(x)$ holds true where ${\bf G} := (G_{\bf k}(x))_{{\bf k}\in {\cal K}}$. Note that $\xi_{j;\bullet}$ is short for $(\xi_{j;{\bf k}})_{{\bf k} \in {\cal K}}$. The term ${\cal V}$ is conditionally centered given $X_1,\ldots,X_n$ and its conditional variance has the upper bound 
\begin{align} \label{eq:V}
\sigma^2 \cdot {\bf e}_0^\dagger ({\cal M}+{\cal S})^{-1} {\cal M} ({\cal M}+{\cal S})^{-1} {\bf e}_0 & \, \leq \, \sigma^2 \cdot {\bf e}_0^\dagger ({\cal M}+{\cal S})^{-1} {\bf e}_0\,,
\end{align}
as the matrix ${\cal M}$ is symmetric. First, in order to evaluate the right side in (\ref{eq:V}), we replace the matrix ${\cal M}$ by its expected value (rescaled by $1/n$) and remove the ridge regularization. Writing ${\cal M}_n := E {\cal M} / n \, = \, \big(E 1_{{\cal N}}(X_1) \cdot \xi_{1;{\bf k}} \cdot \xi_{1;{\bf k}'}\big)_{{\bf k},{\bf k}' \in {\cal K}}$, we study the term
\begin{equation} \label{eq:u0} u_0 = {\bf e}_0^\dagger {\cal M}_n^{-1} {\bf e}_0\,,  \end{equation}
where $u := (u_{{\bf k}})_{{\bf k} \in {\cal K}}$ satisfies ${\cal M}_n u = {\bf e}_0$. For some arbitrary norm $\|\cdot\|_\lambda$ on $\mathbb{C}^J$ we write $X_{1,[J]}^* := \big(\langle X_1^*,\varphi_j\rangle\big)_{j\leq J}$ where $X_1^* := X_1-x$ and $X_{1,[J]}^{*,1} := X_{1,[J]}^* / \|X_{1,[J]}^*\|_\lambda$. Moreover we define
$$ \Phi_J^* := P\big[X_1\in{\cal N} \mid X_{1,[J]}^{*,1}\big]\,, $$
and the random probability measure $P_{J,\delta}\big(\cdot;X_{1,[J]}^{*,1}\big)$ by
$$ P_{J,\delta}\big(A;X_{1,[J]}^{*,1}\big) \, := \, E\big( 1_{{\cal N}}(X_1) \cdot 1_A(\|X_{1,[J]}^*\|_\lambda) \mid X_{1,[J]}^{*,1}\big) \, / \, \Phi_J^* \,, $$
for all Borel sets $A \subseteq \mathbb{R}$. We impose that, on the event $\{\Phi_J^*>0\}$, the measure $P_{J,\delta}\big(\cdot;X_{1,[J]}^{*,1}\big)$ has a Lebesgue density $p_{J,\delta}\big(\cdot;X_{1,[J]}^{*,1}\big)$ which satisfies
\begin{equation} \label{eq:pJ}
P\Big[X_1 \in {\cal N}, \, \sup_{t>0} \, t \cdot p_{J,\delta}\big(t;X_{1,[J]}^{*,1}\big) \, \leq \, c_1 \cdot J\Big] \, \geq \, P[X_1 \in {\cal N}] \, / \, c_2\,,
\end{equation}
for some deterministic positive constants $c_1$ and $c_2$ which do not depend on $\delta$, $n$ or $J$. In the following lemma an upper bound on $u_0$ is provided. 
\begin{lem} \label{L:u0}
Grant the assumption (\ref{eq:pJ}). Then, for $J\geq 2$, 
$$ u_0 \, \leq \, c_2 \cdot \exp\big(8 c_1 \cdot (K-1)\big) \cdot J^{(8c_1+2)\cdot (K-1)} \, / \, P[X_1\in {\cal N}]\,.$$
\end{lem}
Note that, for small $\delta>0$, the term $P[X_1\in {\cal N}]$ is often referred to as small ball probability. We present a key example, which illustrates the constraints of Lemma \ref{L:u0}.
\begin{ex} \label{ex1} {\bf(Gaussian functional covariates) }We consider an example of a functional covariate $X_1$. We apply the approach
\begin{equation} \label{eq:Gauss0} X_1 \, = \, y + \sum_{j=1}^\infty \lambda_j^{1/2} \eta_j \varphi_j\,, \end{equation}
where $y \in L_2([0,1])$ and the sequence $(\lambda_j)_j \downarrow 0$ with $\sum_j \lambda_j < \infty$ are deterministic; and the $\eta_j$ are i.i.d. Gaussian random variables with mean $0$ and variance $1$. Thus,
$$ X_1^* \, = \, z + \sum_{j=1}^\infty \lambda_j^{1/2} \eta_j \varphi_j\,, $$
where $z := y-x$. 



We write ${\bf \eta} := (\eta_j)_{j\leq J}$, ${\bf z} := (\langle z,\varphi_j \rangle)_{j\leq J}$ and $\Lambda$ for the $J\times J$-diagonal matrix which contains $\lambda_j$ as its $(j,j)$th entry. Then, $X_{1,[J]}^* = \Lambda^{1/2} {\bf \eta} + {\bf z}$. Note that $(\varphi_j,\lambda_j)$, $j\geq 1$, form the principal components of $X_1$ and $X_1^*$.
\end{ex}
In the following proposition we show that the assumption of Lemma \ref{L:u0} is satisfied by the functional covariate (\ref{eq:Gauss0}) from Example \ref{ex1}.
 \begin{prop} \label{P:1}
Put $\|{\bf v}\|_\lambda := |\Lambda^{-1/2}{\bf v}|$, for all ${\bf v} \in \mathbb{C}^J$, where $|\cdot|$ denotes the Euclidean norm. Then the functional random variable $X_1$ in (\ref{eq:Gauss0}) satisfies the condition (\ref{eq:pJ}) for $c_1=1+2c_2^*$ and $c_2 = 1$ under the constraint 
\begin{align} 
\label{eq:prop2} & \sup_j \, \langle z,\varphi_j \rangle^2 / \lambda_j \, \leq \, c_2^* \cdot \delta^2\,,
\end{align}
for some constant $c_2^*>0$.
\end{prop}



Next, we provide an upper bound on the expectation of the right side of (\ref{eq:V}) where we use Lemma \ref{L:u0} and the random proximity between the matrix ${\cal M}/n$ and its expected value ${\cal M}_n$ for large $n$.
 \begin{lem} \label{L:V}
For $\delta \in (0,1)$, it holds that
$$ E \, {\bf e}_0^\dagger ({\cal M} + {\cal S})^{-1} {\bf e}_0 \, \leq \, \frac{2-\delta^2}{1-\delta^2} \cdot  n^{-1} \cdot u_0\,, $$
with $u_0$ as in (\ref{eq:u0}).
\end{lem} 
We establish an upper bound on the term ${\cal B}_1$ in (\ref{eq:errt}) by the next lemma.
\begin{lem} \label{L:B1}
We have
$$ {\bf G}^\dagger {\cal S} {\bf G} \, \leq \, \sum_{k=0}^{K-1} \frac{J^k}{k!^2} \|g^{(k)}(x;\cdot)\|^2\,, $$
with $G_{{\bf k}}(x)$ as in (\ref{eq:not}). 
\end{lem}
Applying the Cauchy-Schwarz inequality and the positive semi-definiteness of ${\cal M}$ we deduce from Lemma \ref{L:V} and \ref{L:B1} that
\begin{align} \nonumber 
E\big|{\cal B}_1\big|^2 & \, \leq \, {\bf G}^\dagger {\cal S} {\bf G} \cdot E {\bf e}_0^\dagger ({\cal M}+{\cal S})^{-1} {\bf e}_0 \\
\label{eq:B1}  
& \, \leq \, \frac{2-\delta^2}{1-\delta^2} \cdot  n^{-1} \cdot u_0 \cdot  \sum_{k=0}^{K-1} \frac{J^k}{k!^2} \|g^{(k)}(x;\cdot)\|^2\,. 
\end{align}
Now we focus on the terms ${\cal B}_2$ and ${\cal B}_3$ and impose that $\delta$ has an upper bound smaller than $1$. Using the Cauchy-Schwarz inequality with respect to the sum over $j=1,\ldots,n$; and (\ref{eq:V}), we have that 
$$ |{\cal B}_3|^2 \, \leq \, {\bf e}_0^\dagger ({\cal M}+{\cal S})^{-1} {\bf e}_0 \, \sum_{j=1}^n 1_{{\cal N}}(X_j) \cdot \big({\cal R}_S^{(K)}(x,X_j)\big)^2\,. $$ 
Then, by Lemma \ref{L:u0} and \ref{L:V}, along with the bound (\ref{eq:remain}), we conclude that 
\begin{equation} \label{eq:B3}
 |{\cal B}_3|^2  \leq {\cal O}_P\Big(\exp\big(8c_1\cdot (K-1)\big) \cdot J^{(8c_1+2)\cdot (K-1)}\cdot \delta^{2K} \sup_{y\in {\cal N}} \big\|g^{(K)}(y;\cdot)\big\|^2 / (K!)^2\Big),
\end{equation}
holds true under (\ref{eq:pJ}). The term $|{\cal B}_2|^2$ is analogously bounded from above while Lemma \ref{L:bounddim} is used. We impose (\ref{eq:pJ}), again, and that
\begin{equation} \label{eq:B2cond}
\sum_{k=1}^{K-1} \frac{\delta^{k-1}}{(k-1)!} \cdot \big\|g^{(k)}(x;\cdot)\big\| \, \leq \, c_3\,, \end{equation}
for some universal constant $c_3 \in (0,\infty)$. Then,
\begin{equation} \label{eq:B2}
 |{\cal B}_2|^2  \leq {\cal O}_P\Big(\exp\big(8c_1\cdot (K-1)\big) \cdot J^{(8c_1+2)\cdot (K-1)} \cdot \sum_{j>J} \langle \varphi_j, \Gamma_{{\cal N}} \varphi_j\rangle\Big),
\end{equation}
where $\Gamma_{{\cal N}}$ denotes the linear operator from $L_2([0,1])$ to itself with
$$ \Gamma_{{\cal N}} \, := \, E \big( X_1^* \otimes X_1^* \mid X_1 \in {\cal N}\big)\,, $$
where $\otimes$ denotes the outer product in $L_2([0,1])$. The operator $\Gamma_{{\cal N}}$ has to be studied with an eye on (\ref{eq:B2}). To gain some intuition we reprise the functional random variable (\ref{eq:Gauss0}) from Example \ref{ex1}.
\begin{prop} \label{P:2}
Let $X_1$ be the Gaussian functional random variable from (\ref{eq:Gauss0}). Then, for any integer $j\geq 1$, it holds that 
$$ \langle \varphi_j, \Gamma_{{\cal N}} \varphi_j\rangle \, \leq \, \langle z,\varphi_j\rangle^2 + \lambda_j\,. $$
\end{prop} 
The results of this section remain valid if $g$ is imposed to be $\beta$-fold Fr{\'e}chet differentiable on ${\cal N}$ for some integer $\beta \leq K$ when $K$ is replaced by $\beta$ in (\ref{eq:remain}) and in (\ref{eq:B3}) accordingly. Just put all (non-existing) Fr{\'e}chet derivatives $g^{(k)}(x;\cdot)$ equal to $0$ for $k>\beta$ in (\ref{eq:tayapp}) so that the decomposition (\ref{eq:app}) still holds true. This reflects in the following theorem. Using the decomposition (\ref{eq:errt}) and piecing together (\ref{eq:V}), the Lemmata \ref{L:u0} and \ref{L:V} as well as (\ref{eq:B1}), (\ref{eq:B3}) and (\ref{eq:B2}), we obtain
\begin{theo} \label{T:1} 
Assume that $g$ is $\beta$-fold Fr{\'e}chet differentiable on ${\cal N}$ for some integer $\beta\leq K$. Impose that $\delta$ obeys an upper bound smaller than $1$; and grant (\ref{eq:pJ}) and (\ref{eq:B2cond}). Then, the estimator $\hat{g}$ in (\ref{eq:hatg}) satisfies
\begin{align*} |\hat{g}(x) - g(x)|^2 \, = \, {\cal O}_P\Big(& \exp\big(8c_1\cdot (K-1)\big) \cdot J^{(8c_1+2)\cdot (K-1)} \\ & \cdot \Big\{\Big( \sigma^2 + \sum_{k=0}^{\beta-1} \frac{J^k}{k!^2} \|g^{(k)}(x;\cdot)\|^2\Big) \, / \, \big( n\cdot P[X_1\in {\cal N}]\big) \\ & \, + \, \delta^{2\beta} \sup_{y\in {\cal N}} \big\|g^{(\beta)}(y;\cdot)\big\|^2 / (\beta!)^2 \, + \, \sum_{j>J} \langle \varphi_j, \Gamma_{{\cal N}} \varphi_j\rangle\Big\}\Big)\,, \end{align*}
uniformly with respect to $g$.
\end{theo}

\section{Polynomial Convergence Rates} \label{PCR}

We derive the convergence rates of the estimator $\hat{g}$ under the following conditions. We impose that $g$ is infinitely fold Fr{\'e}chet-differentiable in some neighbourhood ${\cal N} = {\cal N}(x,\delta)$ for some fixed $\delta\in (0,1)$; and that the Fr{\'e}chet derivatives $g^{(k)}$ satisfy
\begin{equation} \label{eq:PCR.1}
\sup_{k\geq 0} \, \sup_{y\in {\cal N}} \, \big\|g^{(k)}(y,\cdot)\big\| \, \leq \, C_{Fre}\,,
\end{equation}
for some global constant $C_{Fre}$. With respect to the operator $\Gamma_{\cal N}$ and the orthonormal basis $\{\varphi_j\}_j$ we assume that
\begin{equation} \label{eq:PCR.2}
\sum_{j=1}^\infty \exp\big(C_{\Gamma,1} \cdot j^\gamma\big) \cdot \langle \varphi_j , \Gamma_{\cal N} \varphi_j\rangle \, \leq \, C_{\Gamma,2}\, 
\end{equation}
for two positive constants $C_{\Gamma,1}$, $C_{\Gamma,2}$ and $\gamma>0$. In the Gaussian Example \ref{ex1}, the condition (\ref{eq:PCR.2}) is satisfied whenever $z$ satisfies (\ref{eq:prop2}) and  the principal components $(\lambda_j)_j$ decay exponentially fast, i.e. 
$$ \sum_{j=1}^\infty \exp\big(C_{\Gamma,1} \cdot j^\gamma\big) \cdot \lambda_j \, \leq \, C_{\Gamma,2}^*\,, $$ 
with a constant $C_{\Gamma,2}^*$ sufficiently small with respect to $c_2^*$ and $C_{\Gamma,2}$,  thanks to Proposition \ref{P:2}. Moreover we impose that the probability $P[X_1\in {\cal N}]$ has some fixed positive lower bound. We select the parameters $K$ and $J$ such that
\begin{align} \label{eq:KJ}
J  & \, \asymp \, (\log n)^{D_0}\,, \\
K  & \, = \, \lfloor D_1 (\log n) / \log\log n \rfloor\,,
\end{align}
with two constants $D_0,D_1>0$ which remain to be determined. Reconsidering Theorem \ref{T:1} for $\beta=K$, we deduce that
\begin{align*} & \exp\big(8c_1\cdot (K-1)\big) \cdot J^{(8c_1+2)\cdot (K-1)} \cdot \sum_{j>J} \langle \varphi_j,\Gamma_{\cal N} \varphi_j\rangle \\ &  \, \leq \, C_{\Gamma,2} \cdot \exp\big\{8c_1\cdot (K-1) + (8c_1+2)\cdot (K-1) \cdot (\log J) - C_{\Gamma,1} \cdot J^\gamma\big\}\,,
\end{align*}
by (\ref{eq:PCR.2}), where this term tends to zero at a rate faster than $1/n$ whenever 
\begin{equation} \label{eq:PCR.cond1}
\gamma D_0 \, > \, 1\,.
\end{equation}
Moreover, we use Stirling's approximation to show that
\begin{align*} & \exp\big(8c_1\cdot (K-1)\big) \cdot J^{(8c_1+2)\cdot (K-1)} / K!^2  \\
& \, \leq \, \exp\big(8c_1\cdot (K-1) + (8c_1+2)\cdot (K-1) \cdot (\log J) + 2K - 2 K \log K\big)\,,
\end{align*}
which converges to zero at a rate faster than $1/n$ whenever
\begin{equation} \label{eq:PCR.cond2}
(4c_1+1) D_0 \, < \, 1\, \mbox{ and }\, 2D_1\big((4c_1+1)D_0-1\big) \, > \, 1\,.
\end{equation}
Finally, since $\sum_{k=0}^{K-1} J^k / k!^2 \leq e\cdot J^K$, the term
\begin{align*}
& \exp\big( 8c_1 \cdot K + (8c_1+3) \cdot K \cdot (\log J) - \log n\big)\,,
\end{align*}
obeys the upper bound $o\big(n^{-\kappa}\big)$ for any $\kappa > 1 - (8c_1+3)D_0D_1$ whenever
\begin{equation} \label{eq:PCR.cond3}
(8c_1+3)D_0D_1 \, < \, 1\,.
\end{equation}
Note that the conditions (\ref{eq:PCR.cond1})--(\ref{eq:PCR.cond3}) can simultaneously be satisfied when $\gamma$ is sufficiently large. Then Theorem \ref{T:1} yields that
$$ |\hat{g}(x) - g(x)|^2 \, = \, {\cal O}_P(n^{-\kappa})\,.$$
so that polynomial rates are achieved.

\section{Proofs} \label{Proofs}

\begin{proof}[Proof of Lemma \ref{L:bounddim}]

Write $X_j^* := X_j-x$ and $X_j^{*,[J]}$ for the orthogonal projection of $X_j^*$ onto the linear hull of $\varphi_1,\ldots,\varphi_J$; moreover $X_j^{*,\bot} := X_j^* - X_j^{*,[J]}$. A telescoping sum expansion yields that
\begin{align*}
\big|& g^{(k)}(x;X_j^*,\ldots,X_j^*) - g^{(k)}(x;X_j^{*,[J]},\ldots,X_j^{*,[J]})\big| \\ & \, = \, \Big| 
\sum_{r=0}^{k-1} g^{(k)}(x;\underbrace{X_j^{*,[J]},\ldots,X_j^{*,[J]}}_{r-\mbox{fold}},X_j^*,X_j^*,\ldots,X_j^*) \, - \, g^{(k)}(x;\underbrace{X_j^{*,[J]},\ldots,X_j^{*,[J]}}_{(r+1)-\mbox{fold}},X_j^*,\ldots,X_j^*)\Big| \\
& \, \leq \, \sum_{r=0}^{k-1} \big|g^{(k)}(x;\underbrace{X_j^{*,[J]},\ldots,X_j^{*,[J]}}_{r-\mbox{fold}},X_j^{*,\bot},X_j^*,\ldots,X_j^*)\big| \\
& \, \leq \, k \cdot \big\|g^{(k)}(x;\cdot)\big\| \cdot \delta^{k-1} \cdot \|X_j^{*,\bot}\|\,.
\end{align*}
as $\|X_j^{*,[J]}\| \leq \|X_j^*\| \leq \delta$. Considering that $\|X_j^{*,\bot}\|^2 = \sum_{l>J} \langle X_j^*,\varphi_l\rangle^2$ leads to the bound
\begin{align}\label{eq:L:bounddim:1}
    \mathrm{I} \leq k \cdot \big\|g^{(k)}(x;\cdot)\big\| \cdot \delta^{k-1} \cdot \big(\sum_{l>J} \langle X_j^*,\varphi_l\rangle^2 \big)^{1/2}.
\end{align}


\end{proof}

\noindent {\it Proof of Lemma \ref{L:u0}}: First define
$$ U_j \, := \, \sum_{{\bf k}\in {\cal K}} u_{{\bf k}} \cdot \xi_{j,{\bf k}}\,. $$
Moreover we introduce the Hilbert space ${\cal H}$ of the equivalence classes consisting of complex-valued random variables $f$ which are measurable with respect to the $\sigma$-algebra generated by $X_1$; satisfy
$$ \|f\|_{\cal H}^2 \, = \, E 1_{\cal N}(X_1) |f|^2 \, < \, \infty\,, $$
and are indistinguishable with respect to the $\|\cdot\|_{\cal H}$-norm. The corresponding inner product is denoted by
$$ \langle f,g\rangle_{\cal H} \, := \, E 1_{{\cal N}}(X_1) f \cdot \overline{g}\,, \qquad f,g \in {\cal H}\,. $$
Writing ${\bf 1}:=\xi_{1;0}$ we deduce that $\langle{\bf 1},U_1\rangle_{\cal H} = 1$ and $\|U_1\|_{\cal H}^2 = u_0$. Let ${\cal H}_{J,K}$ and ${\cal H}_{J,K}^0$ denote the linear hulls of all $\xi_{1;{\bf k}}$, ${\bf k} \in {\cal K}$ and of all $\xi_{1;{\bf k}}$, ${\bf k} \in {\cal K} \backslash\{0\}$, respectively. As $\langle\xi_{1;{\bf k}},U_1\rangle_{\cal H} = 0$ for all ${\bf k}\in {\cal K}\backslash\{0\}$, the random variable $U_1$ lies in the orthogonal complement of ${\cal H}_{J,K}^0$ with respect to ${\cal H}_{J,K}$, which is at most one-dimensional. Therefore the orthogonal projection of ${\bf 1}\in {\cal H}_{J,K}$ onto ${\cal H}_{J,K}^0$ may be represented by ${\bf 1} - \lambda U_1$ for some $\lambda\in \mathbb{C}$. Since $0 = \langle{\bf 1}-\lambda U_1,U_1\rangle_{\cal H}$ it holds that $u_0 = 1/\lambda$. The squared $\|\cdot\|_{\cal H}$-distance between ${\bf 1}$ and ${\cal H}_{J,K}^0$ equals $\|\lambda U_1\|_{\cal H}^2 = \lambda^2 u_0 = 1/u_0$. We have shown that
$$ u_0 \, = \, 1 / \inf\big\{\|{\bf 1} + h\|_{\cal H}^2 : h\in {\cal H}_{J,K}^0\big\}\,. $$
Note that ${\bf 1}+h$, $h\in {\cal H}_{J,K}^0$, may be written as $Q_{J,K}\big(\langle X_1^*,\varphi_1\rangle,\ldots,\langle X_1^*,\varphi_J\rangle\big)$ where $X_1^*:=X_1-x$ and $Q_{J,K}$ is a $J$-variate polynomial of the degree $\leq K-1$ with $Q_{J,K}(0)=1$. Then we apply the fundamental theorem of algebra with the corresponding factorization to the univariate polynomial 
$$ \hat{Q}_{1,K}(t) \, := \, Q_{J,K}\big(X_{1,[J]}^{*,1}\cdot t\big)\,, \qquad t\in \mathbb{C}\,, $$
which satisfies $\hat{Q}_{1,K}(0)=1$ and has the degree $K^* \leq K-1$. This yields that
\begin{align*} u_0 & \, \leq \, 1 \, / \, E \,  \inf_{\hat{Q}_{1,K}} E \Big( 1_{{\cal N}}(X_1) \cdot  \prod_{k=1}^{K^*} \Big|1 - \big\|X_{1,[J]}^*\big\|_\lambda / |\zeta_k|\Big|^2 \mid X_{1,[J]}^{*,1}\Big) \\
& \, = \, 1 \, / \, E \, \Phi_J^* \, \inf_{\hat{Q}_{1,K}} \, \int \,\prod_{k=1}^{K^*} \big|1 - y / |\zeta_k|\big|^2 p_{J,\delta}\big(y;X_{1,[J]}^{*,1}\big) dy \\
&  \, \leq \, 1 \, / \, E \, 1_{\mathfrak{E}(J,\delta)} \, \Phi_J^* \, \inf_{\hat{Q}_{1,K}} \, \int \,\prod_{k=1}^{K^*} \big|1 - y / |\zeta_k|\big|^2 p_{J,\delta}\big(y;X_{1,[J]}^{*,1}\big) dy\,,  \end{align*}
where $\zeta_k$, $k=1,\ldots,K^*$, denote the complex roots of $\hat{Q}_{1,K}$  and $\mathfrak{E}(J,\delta)$ stands for the event
$$ \mathfrak{E}(J,\delta) \, := \, \Big\{ \sup_{t>0} \, t \cdot p_{J,\delta}\big(t;X_{1,[J]}^{*,1}\big) \, \leq \, c_1 \cdot J\Big\}\,. $$
Note that the $(K-1)$-dimensional random vector $(\zeta_1,\ldots,\zeta_{K^*},0,\ldots,0)$ is measurable with respect to the $\sigma$-algebra generated by $X_{1,[J]}^{*,1}$. By Jensen's inequality we deduce that 
\begin{align} \nonumber
 u_0 & \, \leq \, 1 \, / \, E \, 1_{\mathfrak{E}(J,\delta)} \, \Phi_J^* \cdot \exp\Big(2  \, \inf_{\hat{Q}_{1,K}} \, \sum_{k=1}^{K^*} \, \int \log\big|1 - y / |\zeta_k|\big| \, p_{J,\delta}\big(y;X_{1,[J]}^{*,1}\big) \, dy \Big) \\ \label{eq:Baustelle1} &
 \leq 1 \, / \, E \,  1_{\mathfrak{E}(J,\delta)} \, \Phi_J^* \cdot \exp\Big(2 (K-1) \, \min\Big\{ 0 , \, \inf_{r> 0} \int \log\big|1 - y / r\big| \, p_{J,\delta}\big(y;X_{1,[J]}^{*,1}\big) \, dy \Big\}\Big)\,.
\end{align}
Applying the inequality (\ref{eq:pJ}) the integral in (\ref{eq:Baustelle1}) obeys the following lower bound on the event $\mathfrak{E}(J,\delta)$.
\begin{align*}
\int \log\big|1 -  y / r\big|\, &  p_{J,\delta}\big(y;X_{1,[J]}^{*,1}\big)  dy \\ &  \geq \frac1{r(1-1/J)} \, \int_{r(1-1/J)}^{r(1+1/J)} \big(\log |1-y/r|\big) \cdot p_{J,\delta}\big(y;X_{1,[J]}^{*,1}\big) \, y \, dy \, + \, \log(1/J) \\
&  =  \frac1{1-1/J} \, \int_{1-1/J}^{1+1/J} \big(\log |1-z|\big) \cdot p_{J,\delta}\big(rz;X_{1,[J]}^{*,1}\big) \, rz \, dz \, - \, \log J \\
&  \geq \frac{2 c_1 J}{1 - 1/J} \cdot \frac{1}{J} \cdot \big(\log(1/J) - 1\big) \, - \, \log J \\
&   = - 4 c_1 \, - \, (4 c_1 + 1)\cdot \log J\,,
\end{align*}
as $J\geq 2$. 

Inserting this bound in (\ref{eq:Baustelle1}) we conclude that
\begin{align*}
u_0 & \, \leq \, \exp\big(8 c_1 \cdot (K-1)\big) \cdot J^{(8c_1+2)\cdot (K-1)} \, / \, E\, 1_{\mathfrak{E}(J,\delta)} \Phi_J^* \\  & \, \leq \, c_2 \cdot \exp\big(8 c_1 \cdot (K-1)\big) \cdot J^{(8c_1+2)\cdot (K-1)} \, / \, P[X_1\in {\cal N}]\,,
\end{align*}
using the condition (\ref{eq:pJ}). \hfill $\square$ \\

\noindent {\it Proof of Proposition \ref{P:1}}: First we derive the conditional Lebesgue density of $\|X_{1,[J]}^*\|_\lambda$ given $X_{1,[J]}^{*,1}$. For two probe functions $\Psi_0$ and $\Psi_1$ it holds that
\begin{align*}
E & \Psi_0\big(\|X_{1,[J]}^*\|_\lambda\big) \cdot \Psi_1\big(X_{1,[J]}^{*,1}\big) \\
&  =  \int \Psi_0\big( \big|\Lambda^{-1/2} {\bf u}\big|\big)  \Psi_1\big({\bf u}/\big|\Lambda^{-1/2} {\bf u}\big| \big)  \exp\big\{-({\bf u} - {\bf z})^\dagger \Lambda^{-1} ({\bf u} - {\bf z}) / 2\big\} d{\bf u} \, \big( (2\pi)^J \, \det \Lambda\big)^{-1/2} \\
&  =  \exp\big\{-|\Lambda^{-1/2}{\bf z}|^2/2\big\} \, E \Psi_0(|{\bf \eta}|) \Psi_1\big(\Lambda^{1/2} {\bf \eta} / |{\bf \eta}|\big) \exp\big\{{\bf z}^\dagger \Lambda^{-1/2} {\bf \eta}\big\} \\
&  =  \exp\big\{-|\Lambda^{-1/2}{\bf z}|^2/2\big\} \, E  \Psi_1\big(\Lambda^{1/2} {\bf \eta}_0\big) \\ & \hspace{3.5cm} \cdot \int_0^\infty \Psi_0(s) \exp\big\{{\bf z}^\dagger \Lambda^{-1/2} {\bf \eta}_0\cdot s\big\} s^{J-1} \exp(-s^2/2) ds \cdot 2^{1-J/2} / \Gamma(J/2)\,,
\end{align*}
where ${\bf \eta}_0 := {\bf \eta} / |{\bf \eta}|$. We have used that ${\bf \eta}_0$ and $|\eta|$ are independent and that $|\eta|$ obeys the $\chi^2(J)$-law. By the factorization lemma the conditional expectation of $\Psi_0(\|X_{1,[J]}^*\|_\lambda)$ given $X_{1,[J]}^{*,1}$ may be written as $\tilde{\Psi}_0(X_{1,[J]}^{*,1})$ for some function $\tilde{\Psi}_0$. Now, putting $\Psi_0 \equiv 1$ and changing $\Psi_1$ to $\Psi_1\cdot \tilde{\Psi}_0$, we deduce that
\begin{align*} \tilde{\Psi}_0\big(\Lambda^{1/2} {\bf \eta}_0\big) \, = \, & \int_0^\infty \Psi_0(s) \exp\big\{{\bf z}^\dagger \Lambda^{-1/2} {\bf \eta}_0\cdot s - s^2/2\big\} s^{J-1} ds \\ &  \, \big/ \,  \int_0^\infty \exp\big\{{\bf z}^\dagger \Lambda^{-1/2} {\bf \eta}_0\cdot s - s^2/2\big\} s^{J-1} ds\,, \end{align*}
holds true almost surely so that the conditional density of $\|X_{1,[J]}^*\|_\lambda$ given $X_{1,[J]}^{*,1}$ turns out to be 
\begin{align*}
f\big(t \mid X_{1,[J]}^{*,1}\big) \, = \, & 1_{[0,\infty)}(t) \cdot t^{J-1}  \exp\big\{-\big(t - {\bf z}^\dagger \Lambda^{-1} X_{1,[J]}^{*,1}\big)^2/2\big\} \\ &  \big/ \, \int_0^\infty s^{J-1} \cdot \exp\big\{-\big(s - {\bf z}^\dagger \Lambda^{-1} X_{1,[J]}^{*,1}\big)^2/2\big\} ds\,, \end{align*}
for all $t\in \mathbb{R}$. Hence, for any Borel set $A \subseteq \mathbb{R}$, we obtain that
\begin{align*}
E & \big\{ 1_{{\cal N}}(X_1) \cdot 1_A(\|X_{1,[J]}^*\|_\lambda) \mid X_{1,[J]}^{*,1}\big\} \\
& \, = \, \int_0^\infty 1_A(s) \, F_J\big(\delta^2 - |X_{1,[J]}^{*,1}|^2 s^2\big) \,  s^{J-1}  \exp\big\{-\big(s - {\bf z}^\dagger \Lambda^{-1} X_{1,[J]}^{*,1}\big)^2/2\big\} ds \\ & \hspace{7cm} \big/ \,  \int_0^\infty s^{J-1} \exp\big\{-\big(s - {\bf z}^\dagger \Lambda^{-1} X_{1,[J]}^{*,1}\big)^2/2\big\} ds\,,
\end{align*}
where $F_J$ denotes the distribution function of $\sum_{j>J} (\lambda_j^{1/2} \eta_j + \langle z,\varphi_j\rangle)^2$. Thus, for $t>0$, 
\begin{align*}
t & \cdot p_{J,\delta}\big(t; X_{1,[J]}^{*,1}\big) \, = \, F_J\big(\delta^2 - |X_{1,[J]}^{*,1}|^2 t^2\big) \,  t^J  \exp\big\{-\big(t - {\bf z}^\dagger \Lambda^{-1} X_{1,[J]}^{*,1}\big)^2/2\big\} \\ & \hspace{3.7cm} \big/ \, \int_0^\infty F_J\big(\delta^2 - |X_{1,[J]}^{*,1}|^2 s^2\big) \,  s^{J-1}  \exp\big\{-\big(s - {\bf z}^\dagger \Lambda^{-1} X_{1,[J]}^{*,1}\big)^2/2\big\} ds\,.
\end{align*}
As $F_J$ increases monotonically the following upper bound applies
\begin{align} \nonumber
t & \cdot p_{J,\delta}\big(t; X_{1,[J]}^{*,1}\big) \\ \nonumber
& \, \leq \,  t^J  \exp\big\{-\big(t - {\bf z}^\dagger \Lambda^{-1} X_{1,[J]}^{*,1}\big)^2/2\big\} \, \big/ \, \int_0^t  s^{J-1}  \exp\big\{-\big(s - {\bf z}^\dagger \Lambda^{-1} X_{1,[J]}^{*,1}\big)^2/2\big\} ds \\ \label{eq:tpJ1}
& \, = \, J \, / \, \int_0^1 \exp\big\{\big(1 - s^{2/J}\big)\, t^2 /2 \, - \, \big(1 - s^{1/J}\big) t {\bf z}^\dagger \Lambda^{-1} X_{1,[J]}^{*,1}\big\} ds\,. \end{align}
The term (\ref{eq:tpJ1}) has the upper bound 
$$ J \, / \, \int_0^1 \exp\big\{\big(s^{1/J}-1\big)\, t \, |\Lambda^{-1/2} {\bf z}|\big\} \, ds\,. $$
Note that (\ref{eq:prop2}) implies $|\Lambda^{-1/2} {\bf z}|^2 \leq c_2^*\cdot \delta^2 \cdot J$. Thus, for $t^2 \leq 4 c_2^*J / \delta^2$, we have that
$$ t \cdot p_{J,\delta}\big(t; X_{1,[J]}^{*,1}\big) \, \leq \, J / \int_0^1 s^{2c_2^*} ds \, = \, (1+2c_2^*) \cdot J\,. $$ 
On the other hand, if $t> 2 \sqrt{c_2^*} \sqrt{J} / \delta$, it holds that
\begin{align*}
\big(1  - s^{2/J}\big)\, t^2 /2 \, - \, \big(1 - s^{1/J}\big) t {\bf z}^\dagger \Lambda^{-1} X_{1,[J]}^{*,1} & \, \geq \, t (1-s^{1/J}) \cdot\big\{  t/2 - \sqrt{c_2^*} \delta \sqrt{J}\big\} \\
& \, \geq \, 0\,, \end{align*}
for all $s\in [0,1]$ and $\delta \leq 1$. Therefore, the denominator in (\ref{eq:tpJ1}) is bounded from below by $1$ so that
$$ \sup_{t> 2 \sqrt{c_2^*} \sqrt{J} / \delta} t \cdot p_{J,\delta}\big(t; X_{1,[J]}^{*,1}\big) \, \leq \, J\,, $$
which completes the proof. \hfill $\square$ \\

\noindent {\it Proof of Lemma \ref{L:V}}: Writing $\Delta := {\cal M} - E{\cal M}$ and ${\cal M}_{n,{\cal S}} := E{\cal M} + {\cal S}$ where $E{\cal M} = n {\cal M}_n$ we deduce, by elementary matrix algebra, that
\begin{align} \nonumber
(&{\cal M} + {\cal S})^{-1} - (n {\cal M}_n + {\cal S})^{-1} \, = \, (\Delta + {\cal M}_{n,{\cal S}})^{-1} - {\cal M}_{n,{\cal S}}^{-1} \, = \, - ({\cal M} + {\cal S})^{-1} \, \Delta \, {\cal M}_{n,{\cal S}}^{-1} \\  \nonumber
& \, = \, - \big\{(\Delta + {\cal M}_{n,{\cal S}})^{-1} - {\cal M}_{n,{\cal S}}^{-1}\big\} \, \Delta \, {\cal M}_{n,{\cal S}}^{-1} \, - \, {\cal M}_{n,{\cal S}}^{-1} \, \Delta \, {\cal M}_{n,{\cal S}}^{-1} \\ \label{eq:matrix}
& \, = \, {\cal M}_{n,{\cal S}}^{-1} \, \Delta \, (\Delta + {\cal M}_{n,{\cal S}})^{-1} \, \Delta \, {\cal M}_{n,{\cal S}}^{-1}  \, - \, {\cal M}_{n,{\cal S}}^{-1} \, \Delta \, {\cal M}_{n,{\cal S}}^{-1}\,.
\end{align}
Now multiply the terms in (\ref{eq:matrix}) by ${\bf e}_0^\dagger$ and ${\bf e}_0$ from the left and the right, respectively, and take the expectation thereafter. As ${\cal M}$ is positive semidefinite the result has the upper bound
\begin{equation} \label{eq:matrix2}
{\bf e}_0^\dagger\, {\cal M}_{n,{\cal S}}^{-1} \{ E \Delta {\cal S}^{-1} \Delta\} \, {\cal M}_{n,{\cal S}}^{-1} \, {\bf e}_0 \, = \, n \cdot {\bf e}_0^\dagger\, {\cal M}_{n,{\cal S}}^{-1} \{ E \Delta_1 {\cal S}^{-1} \Delta_1\} \, {\cal M}_{n,{\cal S}}^{-1} {\bf e}_0
\end{equation}
where $\Delta = \sum_{j=1}^n \Delta_j$ and $$\{\Delta_j\}_{{\bf k},{\bf k}'} := 1_{{\cal N}}(X_j) \, \xi_{j;{\bf k}} \xi_{j;{\bf k}'} - E 1_{{\cal N}}(X_j) \, \xi_{j;{\bf k}} \xi_{j;{\bf k}'}.$$
Term (\ref{eq:matrix2}) is bounded from above by
\begin{equation} \label{eq:matrix3}
 n \, {\bf e}_0^\dagger {\cal M}_{n,{\cal S}}^{-1}\, \Big\{E   1_{{\cal N}}(X_1) \, \xi_{1;{\bf k}} \, \xi_{1,{\bf k}''} \, \sum_{{\bf k}'} \xi_{1;{\bf k}'}^2 \, {|{\bf k}'| \choose k'_1,\ldots,k'_J} \Big\}_{{\bf k},{\bf k}''} \, {\cal M}_{n,{\cal S}}^{-1} {\bf e}_0\,.
\end{equation}
Since
\begin{align*}
\sum_{{\bf k}} \xi_{1;{\bf k}}^2 {|{\bf k}| \choose k_1,\ldots,k_J} & \, = \, \sum_{k=0}^{K-1} \sum_{|{\bf k}|=k} {k \choose k_1,\ldots,k_J} \prod_{l=1}^J \langle X_1^* , \varphi_l \rangle^{2k_l} \\ 
& \, \leq \,  \sum_{k=0}^{K-1} |X_1^*|^{2k} \,  \sum_{|{\bf k}|=k} \, {k \choose k_1,\ldots,k_J} \, \prod_{l=1}^J \big(\langle X_1^* , \varphi_l \rangle^2 / |X_{1,[J]}^*|^2\big)^{k_l} \\
& \, = \,  \sum_{k=0}^{K-1} |X_1^*|^{2k} \, \leq \, 1 / \big(1 - |X_1^*|^2\big)\,,
\end{align*}
with $X_{1,[J]}^*$ as in (\ref{eq:pJ}), the term (\ref{eq:matrix3}) is smaller or equal to 
\begin{align} \nonumber
 n \, E 1_{{\cal N}}(X_1) \, \big|{\bf e}_0^\dagger {\cal M}_{n,{\cal S}}^{-1}\, \{\xi_{1;{\bf k}}\}_{{\bf k}}\big|^2 / (1-\delta^2) & \, = \, {\bf e}_0^\dagger {\cal M}_{n,{\cal S}}^{-1} \, \{E{\cal M}\} \,  {\cal M}_{n,{\cal S}}^{-1} {\bf e}_0 / (1-\delta^2) \\  \label{eq:matrix4}
& \, \leq \,  n^{-1} \, u_0 \, / \, (1-\delta^2)\,.
\end{align}
Combining (\ref{eq:matrix4}) with the inequality
$${\bf e}_0^\dagger (n{\cal M}_n + {\cal S})^{-1} {\bf e}_0 \, \leq \, n^{-1} \cdot u_0\,, $$ 
completes the proof. \hfill $\square$ \\

\noindent {\it Proof of Lemma \ref{L:B1}}: We have 
\begin{align*}
{\bf G}^\dagger {\cal S} {\bf G} & \, = \, \sum_{{\bf k}} G_{{\bf k}}^2(x) / {|{\bf k}| \choose k_1,\ldots,k_J} \\ & \, \leq \, \sum_{k=0}^{K-1} \|g^{(k)}(x;\cdot)\|^2 \frac1{k!^2} J^k \,  \sum_{|{\bf k}|=k} \, {k \choose k_1,\ldots,k_J} \prod_{l=1}^J (1/J)^{k_l} \\
& \, = \, \sum_{k=0}^{K-1} \|g^{(k)}(x;\cdot)\|^2 \frac1{k!^2} J^k\,,
\end{align*}
which proves the claim of the lemma. \hfill $\square$ \\

\noindent {\it Proof of Proposition \ref{P:2}}: We have
$$ \langle \varphi_j, \Gamma_{{\cal N}} \varphi_j\rangle \, = \, E \big( \langle X_1^*,\varphi_j \rangle^2 \mid X_1 \in {\cal N}\big)\, = \, E \big( \zeta_j \mid \zeta_j + \zeta_{-j} \leq \delta^2\big)\,, $$
where $\zeta_j := (\lambda_j^{1/2} \eta_j + z_j)^2$, $z_j := \langle z,\varphi_j\rangle$ and $\zeta_{-j} := \sum_{l\neq j} (\lambda_l^{1/2} \eta_l + z_l)^2$. Since $E(\zeta_j\mid \zeta_j\leq t) \leq E \zeta_j$ for all $t \in \mathbb{R}$, the independence between $\zeta_j$ and $\zeta_{-j}$ yields that 
\begin{align*}
E \big( \zeta_j \mid \zeta_j + \zeta_{-j} \leq \delta^2\big) & \, = \, \int E\big(\zeta_j \mid \zeta_j \leq \delta^2 - s\big) \cdot P[\zeta_j \leq \delta^2-s] \, dP_{\zeta_{-j}}(s) \, / \, P\big[\zeta_j+\zeta_{-j} \leq \delta^2\big] \\ & \, \leq \, E\zeta_j \, = \, z_j^2 + \lambda_j\,, \end{align*}
which completes the proof. \hfill $\square$ \\

\noindent {\bf Acknowledgment.} Moritz Jirak is funded by the Austrian Science Fund (FWF), Project 5485. Alexander Meister and Mario Pahl are supported by the Research Unit 5381 (DFG).

\end{document}